\documentclass[11pt,french,english]{amsart}
\usepackage[T1]{fontenc}
\usepackage[latin1]{inputenc}
\usepackage{latexsym}
\usepackage{amsthm}
\usepackage{amssymb}

\makeatletter

\providecommand{\tabularnewline}{\\}

\numberwithin{equation}{section}
\numberwithin{figure}{section}
\theoremstyle{plain}
\newtheorem{thm}{Theorem}
  \theoremstyle{plain}
  \newtheorem{lem}[thm]{Lemma}
  \theoremstyle{definition}
  \newtheorem{defn}[thm]{Definition}
  \theoremstyle{plain}
  \newtheorem{prop}[thm]{Proposition}

\AtBeginDocument{
  
}

\makeatother

\usepackage{babel}
\addto\extrasfrench{\providecommand{\fg}{\ifdim\lastskip>\z@\unskip\fi~\frqq}}

\begin{document}

\title{Quotients of Fano surfaces}

\author{Xavier Roulleau}

\begin{abstract}
Fano surfaces parametrize the lines of smooth cubic threefolds. In this paper, we study their quotients by some
of their automorphism sub-groups. We obtain in that way some interesting surfaces of general type.
\end{abstract}
\maketitle

\section*{Introduction.}

It is classical to study quotients of surfaces by automorphism groups
in order to obtain new surfaces. For example, Godeaux obtained one
of the first surfaces of general type with vanishing geometric genus
by taking the quotient of a quintic hypersurface in $\mathbb{P}^{3}$
by an order $5$ fixpoint free action. \\
In this paper, we study quotients of Fano surfaces. These surfaces
are by definition modular varieties : they parametrize the lines on
smooth cubic threefolds. This modular property allows to understand
them very well. In fact, we can handle the Fano surface $S$ of a
cubic threefold $F\hookrightarrow\mathbb{P}^{4}$ almost like a hypersurface
in $\mathbb{P}^{3}$: we can think of $F$ as giving the equation
of $S$ from which we can read of the properties of the irregular
surface $S$. In particular, we can obtain the classification of the
automorphism groups of these surfaces. In the present paper we study
the minimal desingularisation of the quotients of these surfaces by
some subgroups of automorphisms. We compute their Chern numbers $c_{1}^{2},\, c_{2},$
irregularity $q$ and geometric genus $p_{g},$ their minimality and
their Kodaira dimension $\kappa$.

Using the classification of cyclic groups of prime order acting on cubic threefolds done in \cite{Gonzalez},
we give in the following table the classification of the
minimal desingularisation of the quotients of Fano surfaces by groups
of prime order, and we give examples of quotients by some automorphisms of order $4$ and $15$:

\begin{tabular}{|c|c|c|c|c|c|c|c|c|c|c|}
\hline 
O & Type & $c_{1}^{2}$ & $c_{2}$ & $q$ & $p_{g}$ & $\chi$ & $g$ & Singularities & Min & $\kappa$\tabularnewline
\hline 
2 & I & $18$ & $54$ & $1$ & $6$ & $6$ & $3$ & $27A_{1}$ & yes & 2\tabularnewline
\hline 
2 & II & $12$ & $12$ & $3$ & $4$ & $2$ &  & $A_{1}$ & yes & 2\tabularnewline
\hline 
3 & III(1) & $15$ & $9$ & $3$ & $4$ & $2$ &  &  & yes & 2\tabularnewline
\hline 
3 & III(2) & $15$ & $33$ & $1$ & $4$ & $4$ & $4$ & $9A_{2}$ & yes & 2\tabularnewline
\hline 
3 & III(3) & $6$ & $54$ & $0$ & $4$ & $5$ &  & $27A_{3,1}$ & yes & 2\tabularnewline
\hline 
3 & III(4) & $-3$ & $3$ & $2$ & $1$ & $0$ &  &  & no & 0\tabularnewline
\hline 
4 & IV(1) & $6$ & $18$ & $1$ & $2$ & $2$ & $4$ & $6A_{1}+A_{3}$ & yes & 2\tabularnewline
\hline 
4 & IV(2) & 0 & 36 & 1 & 3 & 3 & 1 & $12A_{1}+3A_{3}$ & yes & 1\tabularnewline
\hline 
5 & V & $9$ & $15$ & $1$ & $2$ & $2$ & $4$ & $2A_{4}$ & yes & 2\tabularnewline
\hline 
11 & XI & $-5$ & $17$ & $0$ & $0$ & $1$ &  & $5A_{11,3}$ & no & $-\infty$\tabularnewline
\hline 
15 & XV & $-4$ & $16$ & $0$ & $0$ & $1$ &  & $5A_{3,1}+2A_{15,4}$ & no & $-\infty$\tabularnewline
\hline
\end{tabular}

The first and second column give the order and type of the automorphism,
the column $g$ is the genus of the fibration onto the Albanese variety
when it is an elliptic curve, the column Singularities gives the number
and type of singularities on the quotient surface,  Min indicates
if the minimal desingularisation surface is minimal. For the surfaces
which are quotient by the following groups $G,$ we obtain:

\begin{tabular}{|c|c|c|c|c|c|c|c|c|c|}
\hline 
$G$ & $c_{1}^{2}$ & $c_{2}$ & $q$ & $p_{g}$ & $\chi$ & $g$ & Singularities & Min & $\kappa$\tabularnewline
\hline 
$(\mathbb{Z}/2\mathbb{Z})^{2}$(type I) & $5$ & $43$ & $0$ & $3$ & $4$ &  & $24A_{1}$ & yes & 2\tabularnewline
\hline 
$S_{3}$ (type I) & $3$ & $45$ & $0$ & $3$ & 4 &  & $27A_{1}$ & yes & 2\tabularnewline
\hline 
$(\mathbb{Z}/3\mathbb{Z})^{2}$ & $5$ & $19$ & $1$ & $2$ & 2 & $2$ & $6A_{2}$ & yes & 2\tabularnewline
\hline 
$\mathbb{D}_{2}$ (type II) & -3 & 3 & 2 & 1 & 0 &  &  & no & 0\tabularnewline
\hline 
$\mathbb{D}_{3}$ (type II) & $0$ & $12$ & $1$ & $1$ & 1 & $1$ & $A_{1}+3A_{2}$ & yes & 1\tabularnewline
\hline 
$\mathbb{D}_{5}$ (type II) & $-2$ & $2$ & $1$ & $0$ & $0$ & $0$ & $A_{1}$ & no & $-\infty$\tabularnewline
\hline 
$S_{3}\times\mathbb{Z}/3\mathbb{Z}$ & $1$ & $23$ & $0$ & $1$ & $2$ &  & $9A_{1}+3A_{2}$ & yes & 2\tabularnewline
\hline
\end{tabular}

Where $\mathbb{D}_{n}$ is the dihedral group of order $2n$. In each
of the cyclic and non-cyclic cases, we obtain surfaces of all Kodaira
dimensions : rational, abelian, minimal elliptic and of general type. 

The rather exceptional fact that Fano surfaces are modular varieties
enables us know exactly which singularities are on the quotient surface.
Moreover, the situation is so good that we can determine the four
invariants $c_{1}^{2},c_{2},q,p_{g}$ separately and then double-check
our computations by using the Noether formula. We use intersection
theory on singular normal surfaces as defined by Mumford in \cite{Mumford}.
In particular this intersection theory is applied in Propositions
\ref{Order 11, Klein} and \ref{pro:order 15} in order to find the
Kodaira dimension of some surfaces, and we think that this has independent
interest.

Although there are a lot of papers on the subject, the fine classification
of surfaces of general type with small birational invariants in not
achieved, in particular for the irregular ones. Let us discuss the
place of the surfaces we obtain in the geography of surfaces of general
type.

The surfaces of type I are discussed in \cite{Takahashi}.
Some examples of irregular surfaces with $p_{g}=4$ and birational
canonical map are discussed in \cite{Catanese3}. Our type II, III(1),
III(2), III(3) surfaces have $p_{g}=4$ too. The surfaces of type
II is discussed in \cite{Catanese3}, the surface III(3) is described
in \cite{Ikeda}, but our examples III(1) and III(2) are, to our knowledge,
new. We think also that the surfaces of type IV(1) and V are new.

The surface coming from the group $G=S_{3}$ is a Horikawa surface
\cite{Horikawa}. The moduli space of surfaces coming from the group
$G=(\mathbb{Z}/3\mathbb{Z})^{2}$ has recently be work out in \cite{Gentile}.

Our last example is a surface with $K^{2}=p_{g}=1$. In \cite{Catanese1}
and \cite{Catanese2}, Catanese study the moduli of such surfaces,
obtaining counter examples to the global Torelli Theorem.

The paper is divided as follows: in the first section, we remind classical
results from intersection theory and computation of invariants of
quotient surfaces, in the second we recall the known facts about Fano
surfaces and in the third and fourth, we compute the invariants of
the resolutions of the quotient surfaces.

\textbf{Acknowledgements.} Part of this research was done during the author stay in Strasbourg University, the Max-Planck
Institute of Bonn and the Mathematisches Forschungsintitut Oberwolfach.The author wishes to thank the referee for its careful reading of the paper and its comments.

\section{Generalities on quotients and intersection theory.}

Let us recall, mainly without proof, some well-known Lemmas for computing
the invariants of the minimal resolution of the quotient of a surface
$S$ by an automorphism group $G$. 

We will use intersection theory of $\mathbb{Q}$-Cartier divisors
on compact normal surfaces as defined by Mumford \cite{Mumford},
a good reference on that topic is Fulton's book \cite{Fulton}. Let
$Y$ be a normal surface and let $g:Z\rightarrow Y$ be a resolution
of the singularities of $Y$. We denote by $C_{i},i\in I$ the irreducible
reduced components of the exceptional curves of $g$. The intersection
matrix $(C_{i}C_{j})_{i,j}$ is negative definite. For a divisor $C$
on $Y$ let $\bar{C}$ be the strict transform on $Z$ of $C$. Let
$g^{*}C$ and $a_{i},i\in I$ be the $\mathbb{Q}$-divisor and the
positive rational numbers uniquely defined by: \[
\bar{C}=g^{*}C-\sum a_{i}C_{i}\]
and the relations $C_{i}g^{*}C=0$ for all $i$. The intersection
number $CC'\in\mathbb{Q}$ of $C$ and $C'$ is defined by $g^{*}Cg^{*}C'$.
It is bilinear and independent of $g$. Let $K_{Y}$ be the canonical
$\mathbb{Q}$-divisor on $Y,$ then $K_{Z}=\bar{K}_{Y}$ and:
\begin{lem}
Let $K_{Z}=g^{*}K_{Y}-\sum a_{i}C_{i}$ be the canonical divisor of
$Z$. Then:\[
K_{Z}^{2}=K_{Y}^{2}+(\sum a_{i}C_{i})^{2}.\]
If all the components $C_{i}$ of an exceptional divisor of the resolution
$Z\rightarrow Y$ are $(-2)$-curves, then $a_{i}=0$ for all of those
$C_{i}$.
\end{lem}
Let us suppose that there exists a smooth surface $S$ and a finite
automorphism group $G$ such that $Y$ is the quotient of $S$ by
$G$ and let $\pi:S\rightarrow S/G=Y$ be the quotient map. For each
reduced divisor $R$ on the surface $S,$ let $H_{R}$ be the isotropy
group of $R$:\[
H_{R}=\{g\in G/g_{|R}=Id_{R}\}.\]
Let $|E|$ denote the order of a set $E$ ; $|H_{R}|$ is the ramification
index of the quotient map $\pi:S\rightarrow S/G$ over $R$. For a
curve $C$ on $S/G,$ we have $\pi^{*}C=\sum_{R\subset\pi^{-1}C}|H_{R}|R$
and for another divisor $C',$ we have: \[
CC'=\frac{1}{|G|}\pi^{*}C\pi^{*}C'.\]
We say that the divisor $C$ on $Y=S/G$ is nef if $CC'\geq0$ for
all curves $C'$. We have:
\begin{lem}
Let $K_{S/G}$ be the canonical $\mathbb{Q}$-divisor on $S/G$. Then:\[
K_{S}=\pi^{*}K_{S/G}+\sum_{R}(|H_{R}|-1)R\]
in particular:\[
K_{S/G}^{2}=\frac{1}{|G|}(K_{S}-\sum_{R}(|H_{R}|-1)R)^{2}.\]
If $K_{S}-\sum_{R}(|H_{R}|-1)R$ is nef, then $K_{S/G}$ is nef. If
$K_{S/G}$ is nef and $K_{Z}=g^{*}K_{S/G},$ then $K_{Z}$ is nef.
\end{lem}
For any integer $n\geq1,$ we define the stratum on $S$:\[
S_{n}=\{s/|Stab_{G}(s)|=n\},\]
where $Stab_{G}(s)$ is the stabilizer of the point $s$ in $S$.
Using the inclusion-exclusion principle and the multiplicativity property
of étale maps for the Euler number $e,$ we obtain:
\begin{lem}
The Euler number of $S/G$ is given by the formula:\[
e(S/G)=\sum_{n\geq1}\frac{n}{|G|}e(S_{n})=\frac{1}{|G|}(e(S)+\sum_{n\geq2}(n-1)e(S_{n})).\]
The Euler number of the minimal resolution $Z\rightarrow S/G$ is
the sum of $e(S/G)$ and the number of irreducible components of the
exceptional curves of $Z\rightarrow S/G$.
\end{lem}
Let $\Omega_{X}^{i}$ be the bundle of holomorphic $i$-forms on a
smooth variety $X$ and let $\omega_{X}=\wedge^{\dim X}\Omega_{X}$.
We denote the irregularity by $q_{X}$ and the geometric genus by
$p_{g}(X)$.
\begin{lem}
Let $Z$ be the minimal resolution of the surface $S/G$. We have:\[
H^{0}(Z,\Omega_{Z}^{i})\simeq H^{0}(S,\Omega_{S}^{i})^{G}.\]
In particular: $p_{g}(Z)=\dim H^{0}(S,\omega_{S})^{G}$ and $q_{Z}=\dim H^{0}(S,\Omega_{S})^{G}$.\end{lem}
\begin{proof}
In \cite{Griffiths} , pp. 349--354, Griffiths gives a definition
of differential forms for singular varieties. This notion coincide
with the usual one when the variety $X$ is smooth and we have $H^{0}(Z,\Omega_{Z}^{i})\simeq H^{0}(X,\Omega_{X}^{i})$
for any resolution of singularities $Z$ of $X$. Moreover, by \cite{Griffiths}
formula (2.8), we have $H^{0}(S/G,\Omega_{S/G}^{i})=H^{0}(S,\Omega_{S}^{i})^{G},$
therefore: $H^{0}(Z,\Omega_{Z}^{i})=H^{0}(S,\Omega_{S}^{i})^{G}$.
\end{proof}

\section{Generalities on Fano surfaces.}

Here we recall the known facts about Fano surfaces. We use mainly
the results of Clemens-Griffiths \cite{Clemens}, Tyurin \cite{Tyurin},
\cite{Tyurin1}, Bombieri Swinnerton-Dyer \cite{Bombieri} and also
\cite{RoulleauFanoSurfaceWith}, \cite{RoulleauGenus2}, \cite{RoulleauElliptic} and \cite{RoulleauKlein}.

Let $S\hookrightarrow G(2,5)$ be the Fano surface parametrizing the
lines on a smooth cubic threefold $F\hookrightarrow\mathbb{P}^{4}$.
The Chern numbers of $S$ are $c_{1}^{2}=45$ and $c_{2}=27$. For
a point $s$ in $S,$ we denote by $L_{s}\hookrightarrow F$ the corresponding
line on $F$. There are $6$ lines through a generic point of $F$.
The closure $C_{s}$ of the incidence :\[
\{t\mid s\not=t,\, L_{t}\, cuts \, L_{s}\}\]
is an ample connected divisor of genus $11$ on $S,$ with at most
nodal singularities. It has the property that if a plane cuts $F$
into three lines $L_{s}+L_{t}+L_{u},$ then $C_{t}+C_{s}+C_{u}$ is
a canonical divisor $K_{S}$. In particular $3C_{s}$ is numerically
equivalent to $K_{S}$ and $C_{s}^{2}=5$.\\
The $5$ dimensional space $H^{0}(\Omega_{S})^{*}$ is the tangent
space of the Albanese variety $Alb(S)$ of $S$. As the Albanese map
of $S$ is an embedding, we therefore consider the tangent space $T_{S,s}$
(for $s$ in $S$) as a subspace of $H^{0}(\Omega_{S})^{*}$. We have:
\begin{thm}
(\cite{Clemens}, Tangent Bundle Theorem $13.37$). There exists an
isomorphism of vector spaces: \[
\phi:H^{0}(\Omega_{S})^{*}\rightarrow H^{0}(F,\mathcal{O}(1))\]
such that for all $s$ in $S$ we have : $H^{0}(L_{s},\mathcal{O}(1))=\phi(T_{S,s})$. 
\end{thm}
In words : the tangent space to the point $s,$ translated to the
point $0$ of $Alb(S),$ is identified by the linear map $\phi$ to
the plane subjacent to the line $L_{s}\hookrightarrow\mathbb{P}^{4}$.
This powerful Theorem has many important consequences:
\begin{thm}
(\cite{RoulleauElliptic}, \cite{RoulleauKlein}). Let $\sigma$ be
an automorphism of $S,$ let $d\sigma$ denotes its action on $H^{0}(\Omega_{S})^{*}$.
The element $d\sigma$ acts naturally on the cubic $F$ and a point
$s$ of $S$ is a fixed point of $\sigma$ if and only if the line
$L_{s}$ is stable under the action of $d\sigma$. \\
If a point $s$ is a fixed point of $\sigma,$ the action of $d\sigma_{s}:T_{S,s}\rightarrow T_{S,s}$
is given by the restriction of $d\sigma$ to the $2$ dimensional
vector space $H^{0}(L_{s},\mathcal{O}(1))$. \\
The map $\sigma\rightarrow d\sigma$ is an isomorphism between
the automorphism groups of $S$ and $F$.
\end{thm}
The following Lemma enables us to compute the geometric genus and
irregularity of the quotient surfaces:
\begin{lem}
(\cite{Clemens},  $(10.14)$). The natural map \ensuremath{\wedge^{2}H^{0}(\Omega_{S})\rightarrow H^{0}(S,\omega_{S})}
 is an isomorphism.
\end{lem}
Let us fix some notations that we will use thereafter:
\begin{defn}
Let $\sigma$ be an automorphism of $S$ and let $a_{1},\dots,a_{k},\, k\leq5$
the eigenvalues of $d\sigma$. We denote by $V_{a_{i}}$ the eigenspace
in $H^{0}(F,\mathcal{O}(1))=H^{0}(\Omega_{S})^{*}$ with eigenvalue
$a_{i},$ and by $\mathbb{P}(V_{a_{i}})\hookrightarrow\mathbb{P}^{4}$
its projectivisation.
\end{defn}
Now we can read on the cubic $F$ which points are fixed by $\sigma$
: they are the stable lines $L$ in $F$ under the action of $d\sigma$.
Let $n$ be the order of the automorphism $\sigma$. In general, we
know the action of $d\sigma$ on $F\hookrightarrow\mathbb{P}^{4},$
thus we know its action on $H^{0}(\Omega_{S})^{*}$ only up to a $n^{th}$
root of unity. But in our previous papers we computed these actions.
Let us give an example. There is an obvious order two automorphism
$\sigma$ acting on the Fano surface of the cubic: \[
F=\{x_{1}^{2}x_{2}+G(x_{2},x_{3},x_{4},x_{5})=0\}.\]
The lines which are stable by $\sigma$ are:\\
i) the lines in the cone intersection of $\{x_{2}=0\}$ and $F,$
parametrized by a smooth plane cubic curve $E\hookrightarrow S,$\\
ii) the 27 lines on the smooth intersection of $F$ by $\{x_{1}=0\}$.\\
The automorphism $d\sigma$ acting on $H^{0}(\Omega_{S})^{*}$
is :\[
f:x\rightarrow(x_{1},-x_{2},-x_{3},-x_{4},-x_{5})\]
or $-f$. Let $s$ be one of the $27$ isolated points of $\sigma$.
We have $T_{S,s}\subset\{x_{1}=0\}$ and as $s$ is an isolated point
of $\sigma,$ $d\sigma_{s}$ acts by $x\rightarrow-x$ on $T_{S,s},$
therefore $d\sigma=f$. Such kind of order $2$ automorphisms are
called of type I ; their trace on $H^{0}(\Omega_{S})$ is $-3$. Let
us recall:
\begin{prop}
(\cite{RoulleauElliptic}, Thm. $13$).\label{pro:Involution of type I}There
is a natural bijection between the set of elliptic curves $E\hookrightarrow S$
on $S$ and the set of involutions $\sigma_{E}$ of type I. The intersection
number of the curves $E,E'$ is given by the formula: \[
EE'=\left\{ \begin{array}{ccc}
-3 & if & E=E'\\
0 & if & o(\sigma_{E}\sigma_{E'})=3\\
1 & if & o(\sigma_{E}\sigma_{E'})=2\end{array}\right.\]
where $o(g)$ denotes the order of an automorphism $g$. \\
If $s$ is a point on $E,$ then $C_{s}=E+F_{s}$ where $F_{s}$
is the fiber over $s$ of a fibration $\gamma_{E}:S\rightarrow E$
invariant by $\mbox{\ensuremath{\sigma}}_{E},$ and such that the
lines $L_{t},L_{\sigma_{E}t},L_{\gamma_{E}t}$ in $F$ are coplanar
for all $t$ in $S$. 
\end{prop}
There is another class of involutions acting on Fano surfaces, called
of type II. The trace of their action on $H^{0}(\Omega_{S})$ equals
$1$. An involution that is the product of two involutions of type
I has type II. 
\begin{prop}
(\cite{RoulleauGenus2}, Thm. $3$).\label{pro: Involution type II}The
fixed point set of an involution of type II is the union of an isolated point $t$
and a smooth genus $4$ curve $R_{t}$. There exists a genus $2$
curve $D_{t}$ on $S$ such that \[
C_{t}=D_{t}+R_{t}\]
The curve $D_{t}$ is smooth or sum of two elliptic curves which intersect
in $t$. Let $\sigma,\,\sigma',...$ be involutions of type II generating
a group such that all involutions have type II. The intersection number
of the curves $R_{t},R_{t'}...$ is given by the formula:\[
R_{t}R_{t'}=\left\{ \begin{array}{ccc}
-3 & if & \sigma=\sigma'\\
1 & if & o(\sigma\sigma')=2\,\, or\,\,6\\
3 & if & o(\sigma\sigma')=3\\
2 & if & o(\sigma\sigma')=5\end{array}\right.\]
where o(f) is the order of the element $f$. For the intersection
$D_{t}D_{t'},$ we have: $D_{t}D_{t'}=R_{t}R_{t'}-1$. 
\end{prop}
We denote by $x_{1},\dots,x_{5}$ a basis of the space $H^{0}(\Omega_{S})$
of global sections of the cotangent sheaf and by $e_{1},\dots,e_{5}$
the dual basis.

\section{Quotients by cyclic groups. }

Let $\sigma$ be an automorphism of a Fano surface $S$. We denote
by $\pi:S\rightarrow S/\sigma$ the quotient map and by $g:Z\rightarrow S/\sigma$
the minimal resolution of $S/\sigma$.

\ensuremath{\blacksquare}
 Let $E\hookrightarrow S$ be an elliptic curve and let $\sigma=\sigma_{E}$
be the corresponding type I involution. The fixed point set of $\sigma$
is the union of the smooth elliptic curve $E$ and $27$ points.
\begin{prop}
\label{Z/2Z type I}The surface $S/\sigma$ contains $27\, A_{1}$
singularities. The resolution $Z$ of $S/\sigma$ is minimal and has
invariants:\[
c_{1}^{2}=18,\, c_{2}=54,\, q=1,\, p_{g}=6.\]
The Albanese variety of $Z$ is $E$ and the natural fibration $Z\rightarrow E$
has genus $3$ fibers.\end{prop}
\begin{proof}
There is a natural fibration $\gamma:S\rightarrow E$ invariant under
$\sigma_{E}$ such that for all $s$ in $E,$ we have $C_{s}=E+F_{s}$
with $F_{s}$ the fiber of $\gamma$ at $s$. The divisor $\pi^{*}K_{S/\sigma}=K_{S}-E=C_{s}+C_{\sigma s}+F_{\gamma s}$
is ample, therefore $K_{Z}=g^{*}K_{S/\sigma}$ is nef and\[
K_{Z}^{2}=K_{S/\sigma}^{2}=\frac{1}{2}(K_{S}-E)^{2}=18.\]
The invariant sub-spaces of $\ensuremath{H^{0}(\Omega_{S})}$ and
$H^{0}(S,\omega_{S})$ by $\sigma_{E}$ have dimension \ensuremath{1}
 and \ensuremath{6,} that implies that $c_{2}=54$. \\
 A fiber $F_{s}$ of $\gamma$ has genus $7$ ; as $F_{s}E=4,$
the quotient fiber $F_{s}/\sigma_{E}$ has genus $3$.
\end{proof}
The fibers of the Albanese fibration of $Z$ are genus $3$ curves.
In \cite{Takahashi}, Takahashi prove that surfaces with $q=1,$ $K^{2}=3p_{g}\geq12$
and Albanese fibers of genus $3$ are canonical i.e.. their canonical
map is birational.

\ensuremath{\blacksquare}
 Let $\sigma$ an involution of $S$ of type II. The fixed point set
of $\sigma$ is the union of a point $t$ and a smooth genus $4$
curve $R_{t}$. 
\begin{prop}
\label{Z/2Z type II} The minimal resolution $Z$ of the quotient
surface $S/\sigma$ is minimal and has invariants :\[
c_{1}^{2}=12,\, c_{2}=12,\, q=3,\, p_{g}=4,\, h^{1,1}=14.\]
\end{prop}
\begin{proof}
The image of $t$ on the surface $Z/\sigma$ is a node. We have:\[
e(Z)-1=\frac{1}{2}(e(S)+1+e(R_{t})).\]
As $e(R_{t})=-6,$ we get $e(Z)=12$. Moreover, we have :\[
K_{Z}^{2}=K_{S/\sigma}^{2}=\frac{1}{2}(K_{S}-R_{t})^{2}=\frac{1}{2}(45-2\cdot9-3)=12.\]
The other invariants are easily computed. Let $D_{t}$ be the residual
divisor such that $C_{t}=D_{t}+R_{t}$. Let $\equiv$ denotes the numerical equivalence. As $K_{S}-R_{t}\equiv 2C_{t}+D_{t}$
is nef, $K_{S/\sigma}$ is nef and $K_{Z}\equiv g^{*}K_{S/\sigma}$ is nef,
therefore $Z$ is minimal.
\end{proof}
A smooth polarisation $\Theta$ of type $(1,1,2)$ on an Abelian Threefold
has the same invariants as the surface $Z,$ see \cite{Catanese3}. 

\ensuremath{\blacksquare}
 Let \ensuremath{\alpha}
 be a primitive third root of unity. Let $\sigma$ be an order $3$
automorphism of $S$ such that the eigenvalues of $d\sigma$ acting
on $H^{0}(\Omega_{S})$ are \ensuremath{\alpha^{2},\alpha,1,1,1}
 (automorphism of type III(1)). 
\begin{prop}
\label{(a^2,a,1,1,1) order 3}The automorphism $\sigma$ has no fixpoints.
The quotient surface $S/\sigma=Z$ is smooth, minimal, and has invariants:
\[
c_{1}^{2}=15,\, c_{2}=9,\, q=3,\, p_{g}=4.\]
\end{prop}
\begin{proof}
Up to a change of coordinates, the cubic can be written as: \[
F=\{x_{1}^{3}+x_{2}^{3}+ax_{1}x_{2}x_{3}+C(x_{3},x_{4},x_{5})=0\}\]
with $C$ a cubic form. As $F$ is smooth, there are no lines into
the intersection of $F$ and the plane $\mathbb{P}(V_{1})$. Moreover,
there are no lines in $F$ going through the points $\mathbb{P}(V_{\alpha})$
and $\mathbb{P}(V_{\alpha^{2}}),$ therefore the automorphism $\sigma$
has no fixed points and the surface $S/\sigma=Z$ is smooth. We have
moreover: $\pi^{*}\ensuremath{K_{Z}=K_{S}},$ thus $K_{Z}$ is ample.
As $\pi$ is étale $K_{Z}^{2}=\frac{1}{3}K_{S}^{2}$ and $c_{2}(Z)=\frac{1}{3}c_{2}(S)$.
Since the action of $\sigma$ on \ensuremath{H^{0}(\Omega_{S})}
 is known, we can compute the other invariants.
\end{proof}
In view of \cite{Catanese3}, where Catanese and Schreyer discuss
about irregular surfaces with $p_{g}=4,$ we collect further informations
on the surface $Z$.

Let $w_{1},\, w_{2}\in H^{0}(\Omega_{S})$ be two linearly independent
$1$-forms on $S$. Recall that by the Tangent Bundle Theorem, the
canonical divisor associated to the form $w_{1}\wedge w_{2}$ parametrizes
the lines on $F\hookrightarrow\mathbb{P}^{4}$ that cut the plane
$\{w_{1}=w_{2}=0\}\hookrightarrow\mathbb{P}^{4}$.

A basis of the $\sigma$-invariant canonical forms on $S$ is $x_{1}\wedge x_{2},\, x_{3}\wedge x_{4},\, x_{3}\wedge x_{5},\, x_{4}\wedge x_{5}$.
Thus, a point $s$ in $S$ is a base point of the corresponding $3$
dimensional linear system if the line $L_{s}$ cuts the $4$ planes:
$x_{1}=x_{2}=0,$ $x_{3}=x_{4}=0,$ $x_{3}=x_{5}=0,$ $x_{4}=x_{5}=0$.
But this is impossible, therefore the system is base point free and
the canonical system of $Z$ too. 

The Albanese map of $Z$ is not a fibration (by \cite{Clemens}, there
is no fibration of a Fano surface onto a curve of genus $>1$). It
would be interesting to study deeper $Z$ in the spirit of \cite{Catanese3},
in particular we can ask if the canonical map is birational.

\ensuremath{\blacksquare}
 Let $\sigma$ be an order $3$ automorphism of $S$ such that the
eigenvalues of $d\sigma$ acting on $H^{0}(\Omega_{S})$ are $(\alpha^{2},\alpha^{2},\alpha,\alpha,1)$
(automorphism of type III(2)). 
\begin{prop}
\label{Trois a^2,a^2,a,a,1}The \ensuremath{9}
 singularities of the quotient \ensuremath{S/\sigma}
 are cusps \ensuremath{A_{2}}.
 The minimal resolution $Z$ of this surface has invariants:\[
c_{1}^{2}=15,\, c_{2}=33,\, q=1,\, p_{g}=4,\, h^{1,1}=27\]
and is minimal. The fibers of the fibration onto the Albanese variety
have genus $4$.\end{prop}
\begin{proof}
Up to a change of coordinates, $\sigma$ acts on the cubic: \[
F=\{x_{1}^{3}+x_{2}^{3}+x_{3}^{3}+x_{4}^{3}+x_{5}^{3}+\ell_{1}(x_{1},x_{2})\ell_{2}(x_{3},x_{4})x_{5}=0\}.\]
The lines $\mathbb{P}(V_{\alpha^{2}})$ and $\mathbb{P}(V_{\alpha})$
and the point $\mathbb{P}(V_{1})$ are not contained on $F$. The
stable lines are the $9$ lines on $F$ that cut the disjoint lines
\ensuremath{\mathbb{P}(V_{\alpha})}
 and \ensuremath{\mathbb{P}(V_{\alpha^{2}})}.
 Let $s$ be one of the fixed points of $\sigma$. The eigenvalues
of $d\sigma$ acting on $T_{S,s}$ are \ensuremath{\alpha,}\ensuremath{\alpha^{2},}
 therefore the image of $s$ on the quotient surface $S/\sigma$ is
a $A_{2}$ singularity, resolved by a chain of $2$ $(-2)$-curves.
We have:\[
K_{Z}^{2}=K_{S/\sigma}^{2}=\frac{45}{3}=15.\]
moreover $e(Z)-9\cdot2=\frac{1}{3}(27+2\cdot9)$ and $e(Z)=33$. \\
Let us compute the genus of the fibers. By \cite{RoulleauFanoSurfaceWith}
Cor. 26, there is a fibration of $S$ onto an elliptic curve $E,$
invariant by $\sigma$ and with fibers $F$ of genus $10$. Therefore
the quotient surface $S/\sigma$ has a fibration $S/\sigma\rightarrow E$
by fibers of genus $4$.
\end{proof}
By using the same method as for surfaces of type III(1), we see that
the canonical system of $Z$ has no has point. Again, it would be
interesting to study deeper $Z$ in the spirit of \cite{Catanese3}.

\ensuremath{\blacksquare}
 Let $\sigma$ be an order $3$ automorphism of $S$ such that the
eigenvalues of $d\sigma$ acting on $H^{0}(\Omega_{S})$ are $(\alpha^{2},\alpha,\alpha,\alpha,\alpha)$
(automorphism of type III(3)). The space $\mathbb{P}(V_{\alpha})$
is a hyperplane, $\mathbb{P}(V_{\alpha^{2}})$ is one point outside
$F$. The hyperplane $\mathbb{P}(V_{\alpha})$ cuts $F$ into a smooth
cubic surface $Y,$ therefore $\sigma$ fixes \ensuremath{27}
 isolated points.
\begin{prop}
\label{(a^2,a,a,a,a) order 3}The quotient $Y/\sigma$ has $27\, A_{3,1}$
singularities. Its minimal resolution $Z$ has invariants \[
c_{1}^{2}=6,\, c_{2}=54,\,\ensuremath{q=0},\,\ensuremath{p_{g}=4},\, h^{1,1}=44.\]

\end{prop}
Up to the change of coordinates, the cubic $F$ has equation $$F=\{x_{1}^{3}+G(x_{2},\dots,x_{5})=0\},$$
with $G$ a cubic form. The surface $Z$ is studied by Ikeda \cite{Ikeda}
; it is the resolution of the double cover of the smooth cubic surface
$Y=\{G=0\},$ ramified along the intersection of $Y$ with its Hessian.
By \cite{Ikeda}, the surface $Z$ is a minimal surface, its canonical
system is base point free and the image of the canonical map is $Y$. 
\begin{proof}
The automorphism $d\sigma$ acts on the hyperplane $V_{\alpha}$ by
multiplication by $\alpha,$ therefore it acts on the tangent space
$T_{S,s}$ of an isolated fixed point $s$ by multiplication by $\alpha$
and the resulting singularity on $S/\sigma$ is a $A_{3,1}.$ We have:\[
K_{S/\sigma}^{2}=\frac{1}{3}K_{S}^{2}=15\]
and $K_{Z}\equiv g^{*}K_{S/\sigma}-\frac{1}{3}\sum_{i=1}^{i=27}E_{i}$ with
$E_{i}$ the $(-3)$-curves over the $A_{3,1}$ singularities. Therefore\[
K_{Z}^{2}=15+\frac{1}{9}27\cdot(-3)=6.\]
The computation of $q=0$ and $p_{g}=4$ is immediate.
\end{proof}
\ensuremath{\blacksquare}
 Let $\sigma$ be an order $3$ automorphism of $S$ such that the
eigenvalues of $d\sigma$ are \ensuremath{(\alpha,\alpha,\alpha,1,1)}
 (automorphism of type III(4)). Then $\mathbb{P}(V_{\alpha})$ is
a plane and $\mathbb{P}(V_{1})$ is a line. The family of lines going
through the plane $\mathbb{P}(V_{\alpha})$ and the line $\mathbb{P}(V_{1})$
is the union of $3$ disjoint elliptic curves.
\begin{prop}
\label{(a,a,a,1,1) order 3}The quotient map $\pi:S\rightarrow S/\sigma$
is a triple cover branched over $3$ elliptic curves of the blow-up
in three points of an abelian surface. \end{prop}
\begin{proof}
See \cite{RoulleauFanoSurfaceWith}.
\end{proof}
\ensuremath{\blacksquare}
 Let $\sigma$ be an order $4$ automorphism of $S$ such that the
eigenvalues of $d\sigma$ acting on $H^{0}(\Omega_{S})^{*}$ are $-1,-1,1,i,-i$. 

\begin{prop}
\label{Order 4}The quotient surface contains $6$ nodes and one $A_{3}$
singularity. The minimal resolution $Z$ of $S/\sigma$ is minimal
and has invariants:\[
c_{1}^{2}=6,\, c_{2}=18,\, q=1,\, p_{g}=2,\, h^{1,1}=16.\]
The fibers of the natural fibration of $Z$ onto its Albanese variety
have genus $4$.
\end{prop}

\begin{proof}
Up to a change of coordinates, the automorphism\[
x\rightarrow(-x_{4},-x_{1},-x_{2},-x_{3},-x_{5})\]
acts on the cubic:  
$$F=\{x_5^3+ax_5^2\sigma_1+x_5(b\sigma_1^2+c\sigma_2)+P(\sigma_1,\sigma_2,\sigma_3)=0\},$$
where $\sigma_i=x_1^i+\dots +x_4^i$ and $P$ is a polynomial such that $P(\sigma_1,\sigma_2,\sigma_3)$ is homogenous of degree $3$ in the variables $x_j$. The basis of
$V_{-1},$ $V_{-i},$ $V_{i}$ and $V_{1}$ are respectively: \ensuremath{e_{1}+e_{2}+e_{3}+e_{4},e_{5},}
 \ensuremath{v_{-i}=e_{1}-ie_{2}-e_{3}+ie_{4},}
 \ensuremath{v_{i}=e_{1}+ie_{2}-e_{3}-ie_{4}}
 and \ensuremath{v_{1}=e_{1}-e_{2}+e_{3}-e_{4}.}
 The line through $\mathbb{C}v_{-i}$ and \ensuremath{\mathbb{C}v_{i}}
 is on $F,$ this is also the unique isolated stable line $L_{t}$
of $\sigma^{2},$ involution of type II. There are three lines on
$F$ are going trough the line $\mathbb{P}(V_{-1})$ and the point
\ensuremath{v_{i}}
 and three other lines through the line $\mathbb{P}(V_{-1})$ and
the point \ensuremath{v_{-i}}. These $6$ lines correspond to the intersection points of $D_{t}$
and $R_{t},$ where $C_{t}=R_{t}+D_{t}$ are as in Proposition \ref{pro: Involution type II}.
Their images on $S/\sigma$ are nodes.\\
As the eigenvalues of $d\sigma$ acting on $\mathbb{C}v_{-i}+\mathbb{C}v_{i}$
are $(-i,i),$ the image of $t$ on $S/\sigma$ is an $A_{3}=A_{4,3}$
singularity, resolved by $3$ $(-2)$ curves. Let us compute the Euler
number:\[
e(Z)-6-3=\frac{1}{4}(e(S)+(e(R_{t})-6)+3\cdot7)\]
thus $e(Z)=18$. The quotient map $\pi$ is ramified with index $2$
over $R_{t},$ thus $K_{S}=\pi^{*}K_{S/\nu}+R_{t}$. The divisor $K_{S}-R_{t}=2C_{t}+D_{t}$
is nef, therefore $K_{S/\sigma}$ is nef and $K_{Z}=g^{*}K_{S/\sigma}$
is nef, thus $Z$ is minimal, moreover: \[
K_{Z}^{2}=K_{S/\sigma}^{2}=\frac{1}{4}(K_{S}-R_{t})^{2}=6.\]
It is immediate to check that $q=1$ and $p_{g}=2$. \\
By \cite{RoulleauFanoSurfaceWith}, Theorem 18, there exists a
$\sigma$-invariant fibration $S\rightarrow E$ onto an elliptic curve
with generic fibers $D$ of genus $13$ and $R_{t}$ is contained
in a fiber, therefore $D\rightarrow D/\sigma$ is étale and the fiber
$D/\sigma$ has genus $4$. 
\end{proof}
In \cite{Takahashi}, Takahashi constructed all canonical surfaces
with $q=1$ and $K^{2}=3p_{g}$ with $p_{g}\geq4$. The fibers of
the Albanese fibration of such surfaces are genus $3$ curves. As
far as the author knows, the above surface $Z$ seems new on the line
of surfaces with $q=1$ and $K^{2}=3p_{g}$ (the fibers of the Albanese fibration of $Z$ have genus $4$).

\ensuremath{\blacksquare}
 Let $\sigma$ be an order $4$ automorphism of $S$ such that the
automorphism $\sigma:x\rightarrow(ix_{1},ix_{2},ix_{3},-ix_{4},x_{5})$
acts on the cubic threefold $F$. 
\begin{prop}
\label{Order IV(2)}The quotient surface $S/\sigma$ contains $12$
nodes and $3$ singularities $A_{3}$. The minimal resolution $Z$
of $S/\sigma$ is a minimal properly elliptic surface with invariants:\[
c_{1}^{2}=0,\, c_{2}=36,\, q=1,\, p_{g}=3.\]
\end{prop}
\begin{proof}
Up to a change of coordinates, the cubic is: \[
F=\{x_{5}^{2}x_{4}+x_{4}^{2}x_{1}+C(x_{1},x_{2},x_{3})=0\}.\]
The point $\mathbb{C}e_{5}$
is the vertex of a cone in $F$ whose basis is an elliptic curve $E\hookrightarrow S$. The automorphism $\sigma^{2}$ is a type I involution, fixing $E$ and $27$ isolated points. 
A point $s$ on $E$ correspond to a line\[
L_{s}=\{(\lambda x_{1}:\lambda x_{2}:\lambda x_{3}:0:\mu x_{5})/(\lambda:\mu)\in\mathbb{P}^{1},\, C(x_{1},x_{2},x_{3})=0\},\]
 and such a line is stable under $\sigma,$ therefore $E$ is fixed
by $\sigma$. The space $\mathbb{C}e_{5}$ is the tangent space to
$E$ in the Albanese variety of $S,$ therefore, as $\sigma$ fixes
$E,$ the automorphism $d\sigma$ is equal to $x\rightarrow(ix_{1},ix_{2},ix_{3},-ix_{4},x_{5})$
(the eigenvalues of $d\sigma$ acting on $H^{0}(\Omega_{S})^{*}$
are $i,i,i,-i,1$). The line through $\mathbb{P}(V_{1})$ and $\mathbb{P}(V_{-i})$
is not on $F$ and there are $3$ lines going through $\mathbb{P}(V_{-i})$
and that cut the plane $\mathbb{P}(V_{i}),$ these three lines are
among the $27$ isolated fixed lines of $\sigma^{2}$ and give $3A_{3}$
singularities. The images on $S/\sigma$ of the remaining $24$ isolated
fixed points of $\sigma^{2}$ are $12$ nodes. We have \[
e(Z)-(12+3\cdot 3)=\frac{1}{4}(27+24+3e(E)+3\cdot3)=15\]
and $e(Z)=36$. Let be $\gamma_{E}:S\rightarrow E$ be the $\sigma^{2}$-invariant
fibration, associated to $E$. It is also $\sigma$-invariant. Let
$F_{s}$ be the fiber over $s$. Then $\pi^{*}K_{S/\sigma}=K_{S}-3E$
is numerically equivalent to $3F_{s},$ therefore $K_{S}^{2}=K_{S/\sigma}^{2}=\frac{1}{4}(3F_{s})^{2}=0$.
Moreover, as $F_{s}$ is nef, $K_{Z}$ is nef and then $Z$ is minimal.
The invariants $q=1,\, p_{g}=3$ are readily computed.
\end{proof}
\ensuremath{\blacksquare}
 Let $\sigma$ be an order $5$ automorphism of $S$ such that the
eigenvalues of $d\sigma$ acting on $H^{0}(\Omega_{S})^{*}$ are $(1,\xi,\xi^{2},\xi^{3},\xi^{4}),$
with \ensuremath{\xi}
 a primitive \ensuremath{5^{th}}
 root of unity. 
 
\begin{prop}
\label{Order 5}The quotient surface $S/\sigma$ has $2A_{4}$ singularities.
The minimal resolution $Z$ of $S/\sigma$ is minimal and its invariants
are: \[
c_{1}^{2}=9,\, c_{2}=15,\, q=1,\, p_{g}=2,\, h^{1,1}=13.\]
The general fiber of the Albanese map of $Z$ has genus $4$.
\end{prop}

\begin{proof}
Up to a change of coordinates, the cubic is given by:\[
F=\{x_{1}^{2}x_{3}+x_{3}^{2}x_{4}+x_{4}^{2}x_{2}+x_{2}^{2}x_{1}+x_{5}(ax_{1}x_{4}+bx_{2}x_{3})+x_{5}^{3}=0\},\]
and $\sigma$ acts by:\[
x\rightarrow(\xi x_{1},\xi x_{2}^{2},\xi^{3}x_{3},\xi^{4}x_{4},x_{5}).\]
The eigenspace $V_{\xi^{k}}$ is generated by $e_{k}.$ The lines
through two points \ensuremath{\mathbb{C}e_{k},\mathbb{C}e_{k'}}
 and contained in \ensuremath{F}
 are the line \ensuremath{L_{w}}
 through \ensuremath{\mathbb{C}e_{1}}
 and \ensuremath{\mathbb{C}e_{4}}
 and the line \ensuremath{L_{w'}}
 through \ensuremath{\mathbb{C}e_{2}}
 and \ensuremath{\mathbb{C}e_{3}.}
 The eigenvalues of \ensuremath{d\sigma_{w}:T_{S,w}\rightarrow T_{S,w}}
 are \ensuremath{\xi^{4}}
 and \ensuremath{\xi,} the eigenvalues of \ensuremath{d\sigma_{w'}:T_{S,w'}\rightarrow T_{S,w'}}
 are \ensuremath{\xi^{3}}
 and \ensuremath{\xi^{2},} therefore the images of $w,w'$ on $S/\sigma$ are $2A_{4}$ singularities.
We have :\[
e(Z)-2\cdot4=\frac{1}{5}(e(S)+4\cdot2)\]
thus $e(Z)=15$. As $K_{S/\sigma}$ is ample, $K_{Z}$ is nef. Moreover
$K_{Z}^{2}=\frac{1}{5}K_{S}^{2}=9$. By \cite{RoulleauGenus2}, Thm.
3 D), there exists a fibration $S\rightarrow\mathbb{E}$ onto an
elliptic curve $\mathbb{E}$ that is invariant by $\sigma$ and with
fibers of genus $16$. We deduce that the general fiber of the Albanese
map of $Z$ has genus $4$.
\end{proof}

\ensuremath{\blacksquare}
 Let $Z$ be the resolution of the quotient of the Fano surface $S$ of
the Klein cubic: \[
F=\{x_{1}^{2}x_{2}+x_{2}^{2}x_{4}+x_{4}^{2}x_{3}+x_{3}^{2}x_{5}+x_{5}^{2}x_{1}=0\}\]
by the order $11$ automorphism $\sigma$ acting on $F$ by:\[
x\rightarrow(\xi x_{1},\xi^{9}x_{2},\xi^{3}x_{3},\xi^{4}x_{4},\xi^{5}x_{5})\]
where $\xi$ is a $11^{th}$ primitive root of unity ($S$ is unique to have an order $11$ automorphism, see \cite{RoulleauKlein}).

\begin{prop}
\label{Order 11, Klein}The invariants of the surface $Z$ are:\[
c_{1}^{2}=-5,\, c_{2}=17,\, q=p_{g}=0.\]
The surface $S/\sigma$ contains $5$ singularities $A_{11,3}$. \end{prop}
\begin{proof}
Let us denote by $L_{ij}$ the line $x_{s}=x_{t}=x_{u}=0$ where $\{i,j,s,t,u\}=\{1,2,3,4,5\}$.
The lines on $F$ that are stable by $\sigma$ are $L_{13},L_{23},L_{25},L_{45},L_{14}$.
The $5$ corresponding fix points $s_{ij}$ on $S$ give $5$ singularities
$A_{11,3}$ on $S/\sigma$ resolved by curves $A_{ij},B_{ij}$ with
$(A_{ij})^{2}=-3,\,(B_{ij})^{2}=-4$ and $A_{ij}B_{ij}=1$. We have:\[
K_{Z}=g^{*}K_{X}-\frac{1}{11}\sum(6A_{ij}+7B_{ij})\]
with $g^{*}K_{X}^{2}=K_{X}^{2}=\frac{45}{11},$ thus $K_{Z}^{2}=-5$.
The Euler number is:\[
e=\frac{1}{11}(27+10\cdot5)+10=17.\]
Moreover, we check immediately that the invariants subspaces of $H^{0}(S,\Omega_{S})$
and $H^{0}(S,\omega_{S})$ by $\sigma$ are trivial, therefore the
quotient surface has $q=p_{g}=0$.\end{proof}
\begin{prop}
The surface $Z$ is rational.\end{prop}
\begin{proof}
Let us prove the existence of a smooth rational curve $C$ such that
$C^{2}=0$ on a blow-down of $Z$.\\
 Let us denote by $C_{ij}$ the incidence divisor for the stable
line $L_{ij}$ corresponding to the fixed point $s_{ij}$. The automorphism
$\sigma$ acts on $C_{ij}$. Using the equation of $F,$ we see that
among the $5$ fixed points of $\sigma,$ the curve $C_{13}$ contains
$s_{14}$ and $s_{23}$ and moreover, as the line $L_{13}$ is double
(there is a plane $X$ such that $XF=2L_{13}+L_{14}$), the point
$s_{13}$ is on $C_{13}$.\\
 The permutation $\tau=(1,2,4,3,5)$ acts on the Klein cubic threefold
and with the order $11$ automorphism $\sigma,$ it generates an order
$55$ group such that the group generated by $\sigma$ is distinguished.
By these order $5$ symmetries, we therefore know which fixed points
of $\sigma$ are on the curve $C_{ij}$ ($\tau$ acts on the indices
of the $C_{ij},s_{ij}$ etc...). Any incidence divisor $C$ is a double
cover of a plane quintic $\Gamma$ that can be explicitly
computed using \cite{Bombieri}, equation $(6)$. For the divisor
$C_{13},$ the corresponding quintic $\Gamma$ has equation:\[
4x_{2}^{3}x_{4}x_{5}-x_{2}x_{4}^{4}-x_{5}^{5}=0\]
in the plane with coordinates $x_{2},x_{4},x_{5}$. The curve $\Gamma$
has only one nodal singularity, and therefore by \cite{Bombieri}
Lemma $2,$ the curve $C_{13}$ has only one nodal singularity. By
using the order $5$ symmetry $\tau,$ the same property holds for
the all the $C_{ij}$. Let $D_{ij}$ be the reduced image by the quotient
map $\pi$ of $C_{ij}$ and let $\bar{D}_{ij}$ the strict transform
of $D_{ij}$ in $Z$ by the minimal resolution $g:Z\rightarrow S/\sigma$.
We can write:\[
\bar{D}_{13}=g^{*}D_{13}-\frac{1}{11}(a_{13}A_{13}+b_{13}B_{13}+a_{23}A_{23}+b_{23}B_{23}+a_{14}A_{14}+b_{14}B_{14})\]
for $a_{ij},b_{ij}$ rational. Let be $M=\left(\begin{array}{cc}
-3 & 1\\
1 & -4\end{array}\right),$ $M^{-1}=-\frac{1}{11}\left(\begin{array}{cc}
4 & 1\\
1 & 3\end{array}\right)$. We have $\bar{D}_{13}A_{ij},\bar{D}_{13}B_{ij}\in\mathbb{Z}^{+},$
therefore $(a_{ij},b_{ij})M\in(11\mathbb{Z}^{-},11\mathbb{Z}^{-}),$
thus $a_{ij},b_{ij}$ are positive integers. Using the order $5$
symmetry, we get:\[
\bar{D}_{25}=g^{*}D_{25}-\frac{1}{11}(a_{13}A_{25}+b_{13}B_{25}+a_{23}A_{45}+b_{23}B_{45}+a_{14}A_{23}+b_{14}B_{23}).\]
Moreover $\bar{D}_{13}\bar{D}_{25}\in\mathbb{Z}^{+},$ thus\[
\bar{D}_{13}\bar{D}_{25}=\frac{1}{11^{2}}(55+(a_{23},b_{23})M\left(\begin{array}{c}
a_{14}\\
b_{14}\end{array}\right))\in\mathbb{Z}^{+}.\]
Let us define $(a_{23},b_{23})M=(-11u_{1},-11u_{2})$ with $u_{1},u_{2}$
positive integers. We have:\[
\frac{1}{11^{2}}(55-11(a_{14}u_{1}+b_{14}u_{2}))\in\mathbb{Z}^{+}\]
as $a_{14},b_{14}\in\mathbb{Z}^{+},$ we get $\bar{D}_{12}\bar{D}_{25}=0$
and $a_{14}u_{1}+b_{14}u_{2}=5$. Taking care of $\bar{D}_{13}A_{13},\,\bar{D}_{13}B_{13}\in\mathbb{Z}^{+},$
we get the following $8$ possibilities for $(a_{14},b_{14},u_{1},u_{2})$:\[
\begin{array}{c}
(4,1,1,1),\,(1,3,2,1),\,(5,4,1,0),\,(5,15,1,0),\\
(1,3,5,0),\,(4,1,0,5),\,(9,5,0,1),\,(20,5,0,1).\end{array}\]
Thus $(a_{14},b_{14},a_{23},b_{23})$ is one of the following:\[
\begin{array}{c}
(4,1,5,4),\,(1,3,9,5),\,(5,4,4,1),\,(5,15,4,1),\\
(1,3,20,5),\,(4,1,5,15),\,(9,5,1,3),\,(20,5,1,3).\end{array}\]
Using the order $5$ symmetry $\tau=(1,2,4,3,5),$ we get:\[
\bar{D}_{14}=g^{*}D_{14}-\frac{1}{11}(a_{13}A_{14}+b_{13}B_{14}+a_{23}A_{13}+b_{23}B_{13}+a_{14}A_{45}+b_{14}B_{45}).\]
We have $\bar{D}_{13}\bar{D}_{14}\in\mathbb{Z}^{+},$ therefore:\[
\frac{1}{11^{2}}(55+(a_{13},b_{13})M\left(\begin{array}{c}
a_{14}+a_{23}\\
b_{14}+b_{23}\end{array}\right))\in\mathbb{Z}^{+}.\]
Moreover: \[
\left(\begin{array}{c}
a_{14}+a_{23}\\
b_{14}+b_{23}\end{array}\right)=\left(\begin{array}{c}
9\\
5\end{array}\right),\,\left(\begin{array}{c}
10\\
8\end{array}\right),\,\left(\begin{array}{c}
21\\
8\end{array}\right)\, or\,\left(\begin{array}{c}
9\\
16\end{array}\right).\]
As above, let us define $(a_{14}+a_{23},b_{14}+b_{23})M=(-11w_{1},-11w_{2})$
with $w_{1},w_{2}$ positive integers. We obtain : \[
(w_{1},w_{2})=(2,1)\,;\,(2,2)\,;\,(5,1)\, or\,(1,5).\]
As $a_{13}w_{1}+b_{13}w_{2}=5,$ we get the following possibilities
with respect to the 4 above pairs $(w_{1},w_{2})$:\[
\left(\begin{array}{c}
a_{13}\\
b_{13}\end{array}\right)=\left(\begin{array}{c}
2\\
1\end{array}\right),\left(\begin{array}{c}
1\\
3\end{array}\right),\left(\begin{array}{c}
0\\
5\end{array}\right)\,;\,\emptyset\,;\,\left(\begin{array}{c}
1\\
0\end{array}\right),\left(\begin{array}{c}
0\\
5\end{array}\right)\,;\,\left(\begin{array}{c}
0\\
1\end{array}\right),\left(\begin{array}{c}
5\\
0\end{array}\right)\]
but as $\bar{D}_{13}A_{13}\geq0$ is an integer, the only solution
is $(a_{13},b_{13})=(1,3),$ $(w_{1},w_{2})=(2,1)$ and $(a_{14},b_{14},a_{23},b_{23})$
equals $(4,1,5,4)$ or $(5,4,4,1)$. We obtain that: $\bar{D}_{13}^{2}=-1$
and $K_{Z}\bar{D}_{13}=-1$ and by symmetry, the curves $\bar{D}_{ij}$
are $5$ disjoint $(-1)$-curves. Let us suppose that $(a_{14},b_{14},a_{23},b_{23})$
is $(5,4,4,1),$ then: \[
\begin{array}{c}
\bar{D}_{13}=g^{*}D_{13}-\frac{1}{11}(A_{13}+3B_{13}+4A_{23}+B_{23}+5A_{14}+4B_{14})\\
\bar{D}_{25}=g^{*}D_{25}-\frac{1}{11}(A_{25}+3B_{25}+4A_{45}+B_{45}+5A_{23}+4B_{23})\\
\bar{D}_{14}=g^{*}D_{14}-\frac{1}{11}(A_{14}+3B_{14}+4A_{13}+B_{13}+5A_{45}+4B_{45})\\
\bar{D}_{23}=g^{*}D_{23}-\frac{1}{11}(A_{23}+3B_{23}+4A_{25}+B_{25}+5A_{13}+4B_{13})\\
\bar{D}_{45}=g^{*}D_{45}-\frac{1}{11}(A_{45}+3B_{45}+4A_{14}+B_{14}+5A_{25}+4B_{25}).\end{array}\]
We have $A_{13}\bar{D}_{ij}=0,0,1,1,0$ and $B_{13}\bar{D}_{ij}=1,0,0,1,0,$
moreover: \[
\bar{D}_{14}A_{13}=\bar{D}_{14}A_{45}=\bar{D}_{23}A_{13}=\bar{D}_{23}A_{25}=1.\]
The images of $A_{13}$ and $A_{45}$ by the blow-down map of the
five $\bar{D}_{ij}$ are two $(-1)$-curves $A'_{13}$ and $A'_{45}$
such that $A'_{13}A'_{45}=1$ therefore, as $Z$ is regular, it is
a rational surface. In the same way, if we suppose that $(a_{14},b_{14},a_{23},b_{23})$
is $(4,1,5,4),$ we obtain that the surface $Z$ is rational.
\end{proof}

\ensuremath{\blacksquare}
 Let $S$ be the Fano surface of the cubic:\[
x_{1}^{2}x_{3}+x_{3}^{2}x_{4}+x_{4}^{2}x_{2}+x_{2}^{2}x_{1}+x_{5}^{3}=0.\]
The order $15$ automorphism:\[
\sigma:x\rightarrow(\mu x_{1},\mu^{7}x_{2},\mu^{13}x_{3},\mu^{4}x_{4},\mu^{5}x_{5})\]
 ($\mu^{15}=1$) acts on $S$.
\begin{prop}
\label{pro:order 15}The surface $S/\sigma$ contains $ $$5A_{3,1}+2A_{15,4}$
singularities. The minimal resolution $Z$ of $S/\sigma$ has invariants:\[
c_{1}^{2}=-4,\, c_{2}=16,\, q=p_{g}=0,h^{1,1}=14.\]
\end{prop}
\begin{proof}
The automorphism $\sigma$ fixes $2$ isolated points $s_{14},s_{23}$
(corresponding to the lines $\mathbb{C}e_{1}+\mathbb{C}e_{4}$ and
$\mathbb{C}e_{2}+\mathbb{C}e_{3}$) and acts on their tangent spaces
by the diagonal matrix with diagonal elements $(\mu^{4},\mu)$ giving
$2A_{15,4}$ singularities denoted by $a$ and $b$. The singularity
$a$ is resolved by two $(-4)$-curves $T_{a},U_{a}$ such that $T_{a}U_{a}=1,$
the singularity $B$ is resolved by $T_{b},U_{b}$ with the same configuration.
The automorphism $\sigma^{5}$ fixes $27$ isolated points (lines
in the hyperplane $x_{5}=0$) and acts on the tangent space at these
points by multiplication by $\mu^{5}$. The points $s_{14},s_{23}$
are among theses $27$ points. The other $25$ fixed points of $\sigma^{5}$
gives $5A_{3,1}$ singularities on $S/G$ resolved by $5$ $(-3)$-curves
$T_{i}$. We have : $q=p_{g}=0$. For the Euler number:\[
e(S/G)=\frac{1}{15}(27+(3-1)\cdot25+(15-1)\cdot2)=7\]
 and $e(S)=7+5+2\cdot2=16$. For the canonical bundle:\[
K_{Z}=g^{*}K_{S/G}-\frac{1}{3}(2(U_{a}+T_{a})+2(U_{b}+T_{b})+\sum_{i=1}^{i=5}T_{i})\]
thus $K_{Z}^{2}=-4$. \end{proof}
\begin{prop}
The surface $Z$ is rational.\end{prop}
\begin{proof}
In order to prove the Proposition, we will prove the existence of
a smooth rational curve $C$ such that $C^{2}=0$ on a blow-down of
$Z$. The $27$ stable lines under the action of $\sigma^{5}$ are
on the cubic surface $X=F\cap\{x_{5}=0\}$. Their corresponding points
on $S$ are denoted by:\[
e_{1},\dots,e_{6},g_{1},\dots,g_{6},f_{ij},\,1\leq i<j\leq6\]
and their configuration is as follows:\\
The two points $e_{1}$ and $g_{1}$ are fixed by $\sigma$ and
the corresponding lines $L_{e_{1}}$ and $L_{g_{1}}$ ($\mathbb{C}e_{1}\oplus\mathbb{C}e_{4}$
and $\mathbb{C}e_{2}\oplus\mathbb{C}e_{3}$) are skew. The images
of $e_{1}$ and $g_{1}$ on $S/G$ are denoted by $a$ and $b$. The
other points $e_{i}$ and $g_{i}$ are such that $L_{e_{1}},\dots,L{}_{e_{6}},L_{g_{1}},\dots,L_{g_{6}}$
is a double six. The $\{g_{2},\dots,g_{6}\}$ and $\{e_{2},\dots,e_{6}\}$
are orbits of $\sigma$ whose images on $S/G$ are denoted by $f$
and $g$. \\
Each point $f_{ij},\,1\leq i<j\leq6$ on $S$ is the isolated fixed
point of a type II involution that is the product of two type I involution,
therefore each incidence divisor $C_{f_{ij}}$ splits:\[
C_{f_{ij}}=E+E'+R_{ij}\]
for $E,E'$ the two elliptic curves that cut each other in $f_{ij}$
and with $R_{ij}$ the residual divisor. Each of the $10$ elliptic
curves $E$ as above contains exactly $3$ fixed points and these
points are among the $f_{ij}$ (intersection of $E$ by $3$ other
elliptic curves, see \cite{RoulleauElliptic}). \\
We denote by $A$ and $B$ the image of $C_{e_{1}}$. We can denote
by $E_{ij},\,1\leq i<j\leq5$ the ten elliptic curves on $S$. Their
configuration is given by $E_{ij}E_{st}=1$ if $|\{i,j,s,t\}|=4,$
$E_{ij}E_{st}=0$ if $|\{i,j,s,t\}|=3,$ $E_{ij}^{2}=-3$. The divisors
$E_{1}=E_{12}+E_{23}+E_{34}+E_{45}+E_{15}$ and $E_{2}=E_{13}+E_{24}+E_{35}+E_{14}+E_{25}$
are two orbits of $\sigma$ and we denote by $H,L$ their images on
$S/G$.\\
We denote by $m,n,p$ the images of the $15$ points $f_{ij}$
($3$ orbits) : $m=\{f_{12},f_{13},f_{14},f_{15},f_{16}\},$ $n=\{f_{23},f_{34},f_{45},f_{56},f_{26}\},$
$p=\{f_{24},f_{35},f_{46},f_{25},f_{36}\}$.\\
We have $a,g,m\in A,$ $b,f,m\in B$. On $S/\sigma,$ we have:
\[
H^{2}=L^{2}=\frac{1}{15}E_{1}^{2}=\frac{1}{15}E_{2}^{2}=-\frac{1}{3}\]
 and $LH=\frac{1}{15}E_{1}E_{2}=\frac{1}{3}$. The curve $H$ (resp.
$L$) is nodal in $n,$ (resp. $p$) and $H,L$ cut each other in
$m$ transversally. Let $T_{m},T_{n},T_{p}$ be the $(-3)$-curves
over $m,n,p$ and let $\bar{H},\bar{L}$ be the proper transform of
$H,L$. Then:\[
\bar{H}=g^{*}H-\frac{1}{3}(T_{m}+2T_{n}),\,\bar{L}=g^{*}L-\frac{1}{3}(T_{m}+2T_{p})\]
therefore $\bar{H}^{2}=\bar{L}^{2}=-2$ and $\bar{H}\bar{L}=0$. As
$K_{Z}\bar{H}=K_{Z}\bar{L}=0,$ the curves $\bar{H}$ and $\bar{L}$
are two $(-2)$-curves. Since $a,g,m\in A$ and $A$ is nodal in $a,$
we have:\[
\bar{A}=g^{*}A-\frac{1}{3}(T_{a}+U_{a})-\frac{1}{3}(T_{g}+T_{m})\]
where $T_{a}$ and $U_{a}$ are the $2$ $(-4)$-curves over $a$.
Therefore: $\bar{A}^{2}=-1$ and as $K_{Z}\bar{A}=-1,$ $\bar{A}$
is a $(-1)$-curve. We have: \[
\bar{A}\bar{H}=(g^{*}A-\frac{1}{3}(T_{a}+U_{a})-\frac{1}{3}(T_{g}+T_{m}))(g^{*}H-\frac{1}{3}(T_{m}+2T_{n}))=0.\]
In the same way:\[
\bar{B}=g^{*}B-\frac{1}{3}(T_{b}+U_{b})-\frac{1}{3}(T_{f}+T_{m})\]
 is a $(-1)$-curve and $\bar{A}\bar{B}=0,$ moreover:\[
\bar{B}\bar{H}=(g^{*}B-\frac{1}{4}(T_{b}+U_{b})-\frac{1}{3}(T_{f}+T_{m}))(g^{*}H-\frac{1}{3}(T_{m}+2T_{n}))=0.\]
Consider the curves $\bar{A},\bar{B},T_{m},\bar{H},\bar{L}$. They
are smooth rational curves and their intersection matrix is:\[
\left(\begin{array}{ccccc}
-1 & 0 & 1 & 0 & 0\\
0 & -1 & 1 & 0 & 0\\
1 & 1 & -3 & 1 & 1\\
0 & 0 & 1 & -2 & 0\\
0 & 0 & 1 & 0 & -2\end{array}\right).\]
By blowing down four times, we obtain a smooth rational curve $C$
such that $C^{2}=0$. As $q=0,$ the surface $Z$ is rational.
\end{proof}

\section{Quotients by non-cyclic groups.}

Let us now study quotients by non-cyclic groups $G$. We denote by
$\pi:S\rightarrow S/G$ the quotient map and by $g:Z\rightarrow S/G$
the minimal desingularisation. 

\ensuremath{\blacksquare}
 Let $E,E'$ be $2$ genus $1$ curves on $S$ such that $EE'=1$
and let $G\simeq(\mathbb{Z}/2\mathbb{Z})^{2}$ be the group generated
by the involutions of type I $\sigma_{E},\,\sigma_{E'}$. We have:
\begin{prop}
\label{(Z/2Z)^2 type I}The quotient $S/G$ has $24$ nodes. The minimal resolution $Z$ is minimal and has invariants:\[
c_{1}^{2}=5,\, c_{2}=43,\, q=0,\, p_{g}=3,\, h^{1,1}=35.\]
\end{prop}

\begin{proof}
An equation of $F$ is:
$$F=\{x_1^2x_3+x_2^2x_4+G(x_3,x_4,x_5)=0\}.$$
The involution $\sigma_{E}\sigma_{E'}$ has type II and fixes the
intersection point $t$ of $E$ and $E'$ and the divisor $R_{t}$
such that $C_{t}=E+E'+R_{t}$. The involution $\sigma_{E}$ fixes
$E$ and $27$ points, $3$ of them are on $E'$ ; the symmetric situation
holds for $\sigma_{E'}$. The images on $S/G$ of these $1+2\cdot3=7$
isolated fixed points of $G$ are smooth points. The $24$ singular
points of $S/G$ are nodes, and the quotient map is ramified with
order $2$ over the curve $C_{t},$ therefore $\pi^{*}K_{S/G}=K_{S}-C_{t}\equiv2C_{t}$
and $K_{S/G}$ is ample, moreover: \[
K_{Z}^{2}=K_{S/G}^{2}=\frac{1}{4}(K_{S}-C_{t})^{2}=5.\]
The irregularity is $0$ and $p_{g}=3,$ therefore $c_{2}=43$. 
\end{proof}

\ensuremath{\blacksquare}
 Let $\sigma_{1},\sigma_{2}$ be $2$ involutions on type II such
that $\sigma_{3}=\sigma_{1}\sigma_{2}$ is a third involution of type
II. They generate a group $G$ isomorphic to $(\mathbb{Z}/2\mathbb{Z})^{2}=\mathbb{D}_{2}$.

\begin{prop}
\label{(Z/2Z)^2 type II}The surface $Z=S/G$ is smooth and has invariants:\[
c_{1}^{2}=-3,\, c_{2}=3,\, q=2,\, p_{g}=1,\, h^{1,1}=7.\]
The surface $Z$ is the blow up in three points $p_{1},p_{2},p_{3}$
of an abelian surface such that there exist $3$ genus $2$ curves
$R'_{i}$ that cuts each other in the three point $p_{i}$. The map
$\pi:S\rightarrow Z$ is a $(\mathbb{Z}/2\mathbb{Z})^{2}$-cover,
branched over the strict transform of the three curves $R_{i}'$.
\end{prop}

\begin{proof}
An equation of $F$ is:
$$F=\{x_1^2x_4+x_2^2x_5+x_3^2\ell (x_4,x_5)+ax_1x_2x_3+G(x_4,x_5)=0\}.$$
Each involution $\sigma_{i}$ fixes an isolated point $t_{i}$ and
a smooth genus $4$ curve $R_{i}$ and we have $R_{i}R_{j}=1,$ $R_{i}^{2}=-3,$
$K_{S}R_{i}=9$. \\
The action of the group $G$ on $H^{0}(\Omega_{S})$ is generated
by diagonal matrices with diagonal elements $(-1,-1,1,1,1),$ $(1,-1,-1,1,1)$.
Therefore the invariant subspaces $H^{0}(\Omega_{S})^{G}$ and $H^{0}(S,\omega_{S})^{G}$
have dimension $q=2$ and $p_{g}=1$.\\
The lines $x_{3}=x_{4}=x_{5}=0,$ $x_{1}=x_{4}=x_{5}=0,$ $x_{2}=x_{4}=x_{5}=0,$
corresponds to the isolated fix points $t_{i}$ of the $\sigma_{i}$.
For $i\not =j$, the point $t_{i}$ is a non-isolated fixed point of $\sigma_{j}$
because the eigenvalues of $d(\sigma_{j})_{t_{i}}$ acting on $T_{S,t_{i}}$
are $(-1,1)$. That implies that the images on $S/G$ of the $3$
points $t_{i}$ are smooth points and that $S/G=Z$ is smooth. \\
The quotient map $S\rightarrow S/G$ is ramified only over the
$3$ curves $R_{i}$. Let $D_{i}$ be the genus $2$ curve such that
$C_{t_{i}}=D_{i}+R_{i}$. We have :\[
\pi^{*}K_{Z}=K_{S}-R_{1}-R_{2}-R_{3}=D_{1}+D_{2}+D_{3}\]
 and:\[
K_{Z}^{2}=\frac{1}{4}(K_{S}-R_{1}-R_{2}-R_{3})^{2}=\frac{1}{4}(\sum D_{i})^{2}=-3.\]
We deduce that $c_{2}=3$. As $R_{i}R_{j}=1,$ the $3$ divisors $R_{i}$
are the edges of a triangle, with $t_{i}$ the vertex opposite to
the edge $R_{i}$. The involution $\sigma_{i}$ induces an involution
of $R_{j}$ that fixes only two points $t_{i},t_{k}$ ($\{i,j,k\}=\{1,2,3\}$).
The quotient $R'_{i}=R_{i}/G$ has therefore genus $2$. \\
As $\sigma_{i}(t_{j})=t_{j},$ the involution $\sigma_{i}$ acts
on the incident divisor $C_{i}=D_{i}+R_{i}$. As $D_{i}\sum R_{i}=10,$
the curve $D'_{i}=D_{i}/G$ has genus $g$ such that : \[
2=4(2g-2)+10\]
and $g=0$. Moreover $4(D'_{i})^{2}=(\pi^{*}D'_{i})^{2}=D_{i}^{2}=-4$
and $D'_{i}$ is a $(-1)$-curve. As $K_{Z}=\sum D_{i}'$ and  $q=2$, we see
that the minimal model of $Z$ is an Abelian surface. 
\end{proof}

\ensuremath{\blacksquare}
 Let us study the quotient of $S$ by the group $G\simeq(\mathbb{Z}/3\mathbb{Z})^{2}$
generated by by automorphisms $a,b$ such that the eigenvalues of
\ensuremath{da}
 and $db$ are respectively $(\alpha^{2},\alpha,1,1,1)$ and $(1,1,\alpha^{2},\alpha,1)$. 
 
\begin{prop}
\label{(Z/3Z)^2 =00003D<(a^2,a,1,1,1) et  (1,1,a^2,a,1)>}The surface
$S/G$ contains $6$ cusps. The minimal resolution $Z$ of the quotient
surface $S/G$ is minimal and has invariants:\[
c_{1}^{2}=5,\, c_{2}=19,\, q=1,\, p_{g}=2,\, h^{1,1}=17.\]
The fibers of the fibration of $Z$ onto its Albanese variety have
genus $2$.\end{prop}

\begin{proof}

An equation of $F$ is:
$$F=\{x_1^3+x_2^3+x_3^3+x_4^3+ x_5^3+ux_1x_2x_5+vx_3x_4x_5=0\}.$$
The automorphisms $a,a^{2},b,b^{2}$ have no fixed points. The automorphisms
$ab,a^{2}b^{2}$ fix $9$ isolated points, and also $a^{2}b,ab^{2}$.
That gives $6$ cusps singularities on $S/G$. As these singularities
are resolved by $(-2)$-curves, we have:\[
K_{Z}^{2}=K_{S/d}^{2}=\frac{K_{S}^{2}}{9}=5,\]
moreover: $e(Z)-6\cdot3=\frac{1}{9}(27-18),$ therefore $e(Z)=19$.
The invariant subspaces $H^{0}(\Omega_{S})^{G}$ and $H^{0}(S,\omega_{S})^{G}$
are easily computed. \\
By \cite{RoulleauGenus2}, there is a fibration $\gamma:S\rightarrow E$
of $S$ onto an elliptic curve that is invariant by $G$ and with
fibers $F$ of genus $10$. Thus the Albanese fibration of $Z$ has fibers of genus $\frac{1}{2}\frac{1}{9}(F^{2}+FK_{S})+1=2$.
\\
As $K_{S}$ is ample, and $K_{S}=\pi^{*}K_{S/G},$ we see that
$K_{S/G}$ is ample ; as $K_{Z}=g^{*}K_{S/G},$ the canonical divisor
$K_{Z}$ is nef and $Z$ is minimal.
\end{proof}
Recently the moduli space $\mathcal{M}$ of surfaces with $c_{1}^{2}=5,\, q=1,\, p_{g}=2$
has been studied by Gentile, Oliviero and Polizzi \cite{Gentile}.
They give a stratification of $\mathcal{M}$ and prove that $\mathcal{M}$
has at least $2$ irreducible components. It would be interesting
to know in which component and strata the surface $Z$ belongs.

\ensuremath{\blacksquare}
 Let $G$ be the permutation group $S_{3}$ generated by two involutions
$\sigma_{E},\sigma_{E'}$ such that $EE'=0$. The order $3$ automorphism
$\tau=\sigma_{E}\sigma_{E'}$ has no fixed-points (type III(1)). Let
be $E''=\sigma_{E}(E')=\sigma_{E'}(E)$. Let $g:Z\rightarrow S/G$
be the minimal desingularisation of $S/G$.
\begin{prop}
\label{S_3 engendre par Invo de type I}The surface $Z$ is the resolution
of the $27$ nodes on $S/G$ and has invariants:\[
c_{1}^{2}=3,\, c_{2}=45,q=0,\, p_{g}=3,h^{1,1}=31,\]
it is a minimal Horikawa surface with $c_{2}=5c_{1}^{2}+30$.\end{prop}
\begin{proof}
An equation of $F$ is :\[
F=\{x_{1}^{3}+x_{2}^{3}+x_{1}x_{2}\ell(x_{3},x_{4},x_{5})+C(x_{3},x_{4},x_{5})=0\}.\]
Each involution of type I fixes $27$ isolated points and these points
are not fixed by the $2$ other involutions, therefore the surface
$S/G$ contains $27$ nodes. Let $F,F',F''$ be fibers of $\gamma_{E},\gamma_{E'},\gamma_{E''},$
then: \[
K':=F+F'+F''=K_{S}-(E+E'+E'')\]
 is nef ; as $K'=\pi^{*}K_{S/G},$  $K_{S/G}$ is nef thus $K_{Z}=g^{*}K_{S/G}$ is nef. Moreover: \[
K_{Z}^{2}=K_{S/G}^{2}=\frac{1}{6}(K')^{2}=3.\]
It is easy to check that $q=0$ and $p_{g}=3$. Let us compute the
Euler number:\[
e(S/G)=\frac{1}{6}(e(S)+e(E+E'+E'')+3\cdot27)=18\]
As there are $27$ nodes on $S/G,$ $e(Z)=18+27=45$.
\end{proof}
\ensuremath{\blacksquare}
 Let $S$ be a Fano surface and let $\sigma_{1},\sigma_{2}$ be $2$
involutions of type II acting on $S$ and generating a group $G$
isomorphic to the dihedral group $\mathbb{D}_{3},$ with the involution
$\sigma_{1}\sigma_{2}\sigma_{1}$ of type II. 
\begin{prop}
\label{D_3 =00003D<invo type II>}The minimal resolution $Z$ of the
quotient surface $S/G$ has invariants:\[
c_{1}^{2}=0,\, c_{2}=12,\, q=p_{g}=1,\, h^{1,1}=12.\]
It is a minimal properly elliptic surface. The surface $S/G$ contain
$3$ cusps and one node.\end{prop}
\begin{proof}
The representation of $\mathbb{D}_{3}$ on $H^{0}(\Omega_{S})$ splits
into the sum of twice the unique $2$ dimensional irreducible representation
$V_{\frac{1}{3}}$ and the trivial representation $T$ (see \cite{RoulleauGenus2}, Section 3.3 for an equation),
therefore $q=1$. The representation of $\mathbb{D}_{3}$ on $H^{0}(S,\omega_{S})$
is $T+3D+3V_{\frac{1}{3}},$ where $D$ is the determinantal representation,
thus $p_{g}=1$.\\
The element $\sigma=\sigma_{1}\sigma_{2}$ is a type III(2) automorphism
that fixes $9$ points $s_{i}$. There are $3$ involutions of type
II in $G,$ each of them fixes a curve $R_{i}$ and an isolated point
$t_{i}.$ The image of the $t_{i}$ is a $A_{1}$ singularity on $S/\sigma$.
Let $D_{i}$ be the divisor on $S$ such that $C_{t_{i}}=D_{i}+R_{i}$.
We have $R_{i}R_{j}=3$ for $i\not=j$. This gives $3$ fixed points,
say $s_{1},s_{2},s_{3},$ for the whole group $G$ and the images
of the points $s_{4},\dots,s_{9}$ are two cups on $S/\sigma$. The representation
of the group $G$ on the tangent space of points $s_{1},s_{2},s_{3}$
is isomorphic to $V_{\frac{1}{3}}$, their images are smooth points
on the surface $S/G$. As $S/G$ has only nodal singularities or cusps,
we have $K_{Z}^{2}=K_{S/\sigma}^{2}$. By \cite{RoulleauGenus2},
we have $D_{i}D_{j}=2,$ $R_{i}^{2}=-4,$ $K_{S}R_{i}=6,$ and $F=\sum_{i=1}^{i=3}D_{i}$
is a fiber of a fibration $\gamma:S\rightarrow E$ onto an elliptic
curve $E$. As \[
K_{S}-\sum_{i=1}^{i=3}R_{i}=\sum_{i=1}^{i=3}D_{i}=F,\]
we obtain that $K_{Z}^{2}=\frac{1}{6}F^{2}=0$ and we deduce that
$c_{2}=12$. \\
The fibration $\gamma$ is moreover invariant by $\mathbb{D}_{3},$
thus for a generic fiber $F_{s}$ of $\gamma,$ the curve $F_{s}/\mathbb{D}_{3}$
is a fiber of the Albanese map of $Z$. The quotient $F_{s}\rightarrow F_{s}/\mathbb{D}_{3}$
is ramified over $F_{s}\sum R_{i}=F_{s}(K_{S}-F_{s})=18$ points ;
as $K_{S}\sum D_{i}=18,$ the genus of $F/\mathbb{D}_{3}$ is equal
to $1$. \\
As there is a fibration by elliptic curves on $S,$ it has Kodaira
dimension less or equal to $1,$ and since $p_{g}=1,c_{2}=12,$ it
is a minimal properly elliptic surface. 
\end{proof}

\ensuremath{\blacksquare}
 Let $S$ be a Fano surface and let $\mathbb{D}_{5}$ be the dihedral
group acting on it, such that the order $2$ elements have type II. 
\begin{prop}
\label{D_5 =00003D<invo type II>}The minimal resolution $Z$ of the
quotient surface $S/\mathbb{D}_{5}$ has invariants:\[
c_{1}^{2}=-2,\, c_{2}=2,\, q=1,\, p_{g}=0,\, h^{1,1}=4\]
The surface $S/\mathbb{D}_{5}$ contains a unique nodal singularity.
The surface $Z$ is a ruled surface of genus $1$.\end{prop}

\begin{proof}
 We can take the group generated by the permutations $a=(1,2,3,4,5)$
and $b=(1,3)(4,5)$ acting on the basis vectors $e_{1},\dots,e_{5}$
of $\mathbb{C}^{5}$ by permutation of the indices.\\
The vectors $v_{k}=\sum_{k=1}^{k=5}\xi^{ki}e_{i},\,0\leq k\leq4$
are eigenvectors of $a,$ moreover $v_{1}\wedge v_{4}$ and $v_{2}\wedge v_{3}$
are a basis of eigenvectors for the eigenvalue $1$ (resp $-1$) under
the action of $a$ (resp. $b$). The lines corresponding to $s_{1}=\mathbb{C}v_{1}\wedge v_{4}$
and $s_{2}=\mathbb{C}v_{2}\wedge v_{3}$ are the only ones contained
into the cubic $F$ among the points $\mathbb{C}v_{i}\wedge v_{j}$
($1\leq i<j\leq5$) in the grassmannian $G(2,5)$. We deduce that $s_{1}$
and $s_{2}$ are fixed points for the whole group $\mathbb{D}_{5}$.
On the tangent space of $s_{1}$ the action of $\mathbb{D}_{5}$ is
given by $x\rightarrow(\xi x_{1},\xi^{4}x_{2})$ and $x\rightarrow(x_{2},x_{1})$.
The invariant ring by this action is $\mathbb{C}[x_{1}^{5}+x_{2}^{5},x_{1}x_{2}],$
therefore the image of the $s_{i}$ are smooth points. As the eigenvalues
of $db_{s_{1}}$ acting on $T_{S,s_{1}}$ are $(1,-1),$ the fixed
curve of $b$ goes through it. 

The representation of $\mathbb{D}_{5}$ on $H^{0}(\Omega_{S})$ splits
into the trivial representation and the the sum of two $2$ dimensional
non-isomorphic representations $V_{\frac{1}{5}}$ and $V_{\frac{2}{5}},$
therefore $q=1$. The representation of $\mathbb{D}_{5}$ on $H^{0}(S,\omega_{S})$
is $2D+2V_{\frac{1}{5}}+2V_{\frac{2}{5}},$ where $D$ is the determinantal
representation, thus $p_{g}=0$.\\
The group $\mathbb{D}_{5}$ contains $5$ order $2$ elements of
type II, each fixes an isolated point $t_{i}$ and a smooth genus
$4$ curve $R_{i},$ that gives one $A_{1}$ singularity of $S/\mathbb{D}_{5}$.
As $R_{i}R_{j}=2$ for $i\not=j,$ we deduce that the curves $R_{i}$
cut each other in $s_{1}$ and $s_{2}$. Moreover :\[
K_{S}=\pi^{*}K_{S/G}+\sum_{i=1}^{5}R_{i},\]
therefore: \[
10K_{S/G}^{2}=(K_{S}-\sum R_{i})^{2}=-20\]
and as $K_{Z}^{2}=K_{S/G}^{2},$ we obtain: $K_{Z}^{2}=-2$. \\
Let us compute the Euler number:\[
e(S/G)=\frac{1}{10}(e(S)+e(R_{1}+\dots+R_{5}-s_1-s_2)+5+9\cdot2)=1.\]
As we have only one $A_{1}$ singularity : $e(S)=2$. \\
The divisor $F_{t}=\sum_{i=1}^{i=5}D_{i}$ is a connected fiber
of genus $16$ of a fibration $\gamma:S\rightarrow E$ onto an elliptic
curve. This fibration is invariant by $\mathbb{D}_{5}$ (see \cite{RoulleauGenus2},
Theorem 3). The quotient map $F_{u}\rightarrow F_{u}/\mathbb{D}_{5}$
for $F_{u}$ a generic fiber is ramified over $F_{u}\sum R_{i}=F_{t}(5C_{t}-F_{t})=50$
points, thus the genus of the quotient curve $F_{u}/\mathbb{D}_{5}$
is $0$. 
\end{proof}

\ensuremath{\blacksquare}
 Let $S$ be a Fano surface with automorphism group containing two
involutions of type I $\sigma_{E},\sigma_{E'}$ with product of order
$3$ and commuting with a type III(1) automorphism $\sigma$. 
\begin{prop}
\label{S3+Z/3Z}The minimal resolution $Z$ of the quotient surface
$S/G$ is minimal and has invariants:\[
c_{1}^{2}=1,\, c_{2}=23,\, q=0,\, p_{g}=1,\, h^{1,1}=21.\]
\end{prop}
\begin{proof}
Up to a change of coordinates, the cubic can be written as: \[
F=\{x_{1}^{3}+x_{2}^{3}+x_{3}^{3}+x_{4}^{3}+x_{5}^{3}+ax_{1}x_{2}x_{5}=0\}.\]
The fixed points of the three involutions of type I are $3$ disjoint
elliptic curves $E,E',E''$ and $81$ isolated points, divided into
9 orbits of $9$ elements, giving $9A_{1}$ singularities. The $9\cdot2=18$
isolated points of the $2$ type III(2) automorphisms gives $3A_{2}$
singularities on $S/G$. We check easily that $q=0$ and $p_{g}=1,$
moreover:\[
K_{S/G}^{2}=\frac{1}{18}K'^{2}=1\]
for $K'=K_{S}-E-E'-E''$ and we deduce that $c_{2}=23$. As $K'=F+F'+F''$
(for $F,F',F''$ fibers of $\gamma_{E},\gamma_{E'},\gamma_{E''}$)
is nef, $Z$ is minimal.
\end{proof}

\bigskip 
\noindent
Xavier Roulleau,\\
{\tt roulleau@math.ist.utl.pt}\\ 
D\'epartemento de Mathematica,\\ 
Instituto Superior T\'ecnico,\\  
Avenida Rovisco Pais\\ 
1049-001 Lisboa\\ 
Portugal 
\end{document}